\documentclass[12pt]{amsart}
\usepackage{amsfonts, amsmath, cite, comment}
\usepackage{amsmath}
\usepackage{color}

\makeatletter \@addtoreset{equation}{section}

\makeatother \makeatletter

\newtheorem{theorem}{Theorem}[section]
\newtheorem{corollary}[theorem]{Corollary}
\newtheorem{definition}[theorem]{Definition}
\newtheorem{lemma}[theorem]{Lemma}
\newtheorem{proposition}[theorem]{Proposition}
\newtheorem{remark}[theorem]{Remark}

\newcommand{\example}{\vskip 3mm
\goodbreak\refstepcounter{theorem} \noindent
\hbox{{\bf Example \thetheorem.}~}
}

\def\pr1{\prod\hskip -2.07ex * \hskip 0.9 ex}

\newcommand{\ra}{\rightarrow}

\newcommand{\Rr}{\mathbb{R}}

\newcommand{\Div}{\mathrm{Div}}
\renewcommand{\div}{\mathrm{div}}
\renewcommand{\det}{\mathrm{det}}
\newcommand{\Aut}{\mathrm{Aut}}
\newcommand{\id}{\mathrm{id}}

\newcommand{\whp}{\widehat{\partial}}

\usepackage{latexsym} \usepackage{amsmath} \usepackage{amssymb}
\begin{document}

\noindent
\title{On a class of automorphisms in $\mathbb{H}^2$ which resemble
  the property of preserving volume} \thanks{} \author{Jasna Prezelj}
\address{Fakulteta za matematiko in fiziko Jadranska 21 1000
  Ljubljana, Slovenija, UP FAMNIT, Glagolja\v ska 8, Koper Slovenija}
\email { jasna.prezelj@fmf.uni-lj.si} \author {Fabio
  Vlacci}\address{Dipartimento di Matematica e Informatica ``U. Dini''
  - Universit\`a di Firenze Viale Morgagni 67/A, 50134\ Firenze,
  Italy} \email{ vlacci@math.unifi.it} \thanks{\rm The first author
  was partially supported by research program P1-0291 and by research
  projects J1-7256 and J1-9104 at Slovenian Research Agency. Part of
  the paper was written when the first author was visiting the DiMaI
  at University of Florence and she wishes to thank this institution
  for its hospitality.  The second author was partially supported by
  Progetto MIUR di Rilevante Interesse Nazionale PRIN 2010-11 {\it
    Variet\`a reali e complesse: geometria, topologia e analisi
    armonica}.  The research that led to the present paper was
  partially supported by a grant of the group GNSAGA of Istituto
  Nazionale di Alta Matematica ``F: Severi''.}

\begin{abstract}

We give a possible extension for shears and
overshears in the case of two non commutative (quaternionic) variables in relation with the associated vector fields and flows.
We present a possible definition of volume
preserving automorphisms, even though there is no quaternionic
 volume form on $\mathbb{H}^2.$

Using this, we determine a  class of quaternionic automorphisms for which  the
Ander-\-sen-Lempert theory applies. Finally, we exhibit an example of a quaternionic  automorphism, which
is not in the closure of the set of finite compositions of volume preserving quaternionic shears.

\vskip 0.3 cm

{\bf 2010 Mathematics Subject  Classification: 30G35, 58B10}

\end{abstract}
\maketitle

\section{Introduction}

Complex holomorphic shears and overshears represent the major tools
for the description of the groups of automorphisms of $\mathbb{C}^n$
with $n>1$.  In this paper, we give a possible extension for shears and
overshears in the case of two non-commutative variables.  In
particular, we investigate what are the minimal conditions
to define
 good generalizations of the complex
holomorphic  shears and overshears  in relation with the associated vector fields and flows in the non commutative (mainly quaternionic) setting.
To this end, we  restrict our research to mappings represented by convergent quaternionic power series.

Complex analytic shears are simple automorphisms with volume $1.$ Since there does not exist a quaternionic volume form on ${\mathbb H}^n,$ and since the automorphisms with convergent power series as components are not necessarily regular in the sense of \cite{GP}, the class of quaternionic automorphisms with volume $1$ is not defined.

We present an alternative definition of partial derivative, divergence and rotor for the quaternionic setting, and determine the subclasses of vector fields with divergence or rotor. Then, we define automorphisms with volume to be deformations of identity by vector fields with divergence, and we show that they present a proper class of automorphisms  for which  the
Andersen-Lempert theory applies.
In particular, shears and
overshears in this class are the quaternionic analogue of  complex holomorphic shears and overshears.

Finally, we exhibit an example of a quaternionic  automorphism, which
is not in the closure of the set of finite compositions of volume preserving quaternionic shears while its restriction to the complex variables is approximable by a finite composition of (complex) shears.

The paper is structured as follows:  Section $2$ contains the description of our setting with basic definitions and notions, such as partial derivatives, divergence, and rotor. Bidegree full functions are introduced. Section $3$ is devoted to vector fields and their properties, in particular it contains the crucial theorem (Theorem \ref{crcthm}) on vector fields with divergence. Section $4$ studies the connections between Jacobians of shears and  overshears and properties of the corresponding vector fields. Section $5$  presents the application of Andersen-Lempert theory in quaternionic setting with the above-mentioned example.

\section{Preliminaries on convergent quaternionic power series}\label{sec1}

In this section we introduce the basic concepts
and notions to deal with generalizations of complex holomorphic shears
and overshears, flows, and vector fields in the corresponding
quaternionic setting.


 We denote by $\mathbb{H}$ the algebra of quaternions. Let
 $\mathbb{S}$ be the sphere of imaginary quaternions, i.e. the set of
 quaternions $I$ such that $I^2=-1$.  Given any quaternion $z,$ there
 exist (and are uniquely determined) an imaginary unit $I,$ and two
 real numbers $x,y$ (with $y\geq 0$) such that $z=x+Iy$. With this
 notation, the conjugate of $z$ will be
 $\bar z := x-Iy$.  We consider the graded algebra of
 polynomials in the non commutative variables $z_1,\ldots, z_n.$
%
This algebra of polynomials  will be denoted by
$\mathbb{H}[z_1,\ldots,, z_n]$. In other words
\[\mathbb{H}[z_1,\ldots, z_n]=\bigoplus\limits_d  \mathbb{H}_d[z_1,\ldots, z_n]\]
where  $\mathbb{H}_d[z_1,\ldots, z_n]$ consists of finite linear combinations of
 monomials
in the variables $z_1,\ldots,z_n$ of degree $d$ over the  quaternions, namely monomials of the form
\begin{equation}\label{e2}
  a_0 * a_1 * \ldots * a_d, \; a_m \in \mathbb{H},\; \forall m,
\end{equation}
where each $*$ is replaced by one of the variables $z_1,\ldots, z_n.$
Notice that $\mathbb{H}_d[z_1,\ldots, z_n]$ consists of all homogeneous polynomials
in the variables $z_1,\ldots,z_n$ of degree $d$ over the  quaternions.
Our basic assumption on regularity, for the definition of the class of  quaternionic functions
we are interested, in is that any such  function $f$ has a  series expansion of the form
\begin{equation}\label{f}
f(z_1, \ldots, z_n)=\sum_d f_d(z_1,\ldots, z_n)
\end{equation}
with $f_d(z_1,\ldots, z_n)\in\mathbb{H}_d[z_1,\ldots, z_n]$ for any $d$, which converges absolutely.

The set of all such functions -- which turns out to be a right or left
$\mathbb{H}$-module -- will be denoted by $\mathcal{H}[z_1,\ldots,
  z_n]$.  Actually, we can restrict our considerations to the case in
which any $f_d(z_1,\ldots, z_n)$ is a sum of monomials of degree $d$
in the variables $z_1,\ldots,z_n$ whose coefficients $a_0,\ldots,
a_{d-1}$ (using the same notation as in (\ref{e2})) are all in $\mathbb{R}P^3 = S^3/\lbrace -1,1 \rbrace,$ which can be identified with
$\lbrace x = x_0 + x_1 i + x_2 j + x_3 k, \|x\| = 1, x_0 > 0 \mbox{ or } x_0 = 0, x_1 > 0 \mbox{ or } x_0,x_1 = 0, x_2 > 0 \mbox{ or } x = k \rbrace.$ This
fact guarantees formal uniqueness of the expansion in the right
$\mathbb{H}$ module $\mathcal{H}[z_1,\ldots, z_n]$.  We assume the
formal uniqueness of power series expansion of the functions
considered, namely, two such functions are the same iff
the corresponding power series coincide.  Furthermore
$\mathcal{H}[z_1,\ldots, z_n]$ can be considered as a ring with
respect to standard (pointwise) sum and (non commutative)
multiplication.

We remark that  $\mathcal{H}[z_1,\ldots, z_n]$
contains, as a particular case,
the right submodule of slice--regular functions $\mathcal{S}\mathcal{R}$
as introduced in \cite{GP}.
Another interesting subclass of functions in $\mathcal{H}[z_1,\ldots, z_n]$
(which also contains slice--regular functions)
is the one whose elements are functions as in (\ref{f})
such that each of the unitary coefficients $a_0,\ldots,
a_{d-1}$ of $f_d$ is exactly $1$.
This class will be denoted by $\mathcal{H}^{1}[z_1,\ldots, z_n]$.  In
the case of one variable $z_1 = z$ the class $\mathcal{H}^{1}[z] =
\mathcal{S}\mathcal{R}$; the notation $\mathcal{S}\mathcal{R}(D)$
refers to slice--regular functions defined on the open set $D \subset
\mathbb{H}.$ 

In general, there is no standard way of introducing a
notion of (partial) derivative for quaternionic functions (see for
instance \cite{GP,GS}).

We introduce new differential operators
$\widehat{\partial}_{z_j}$ on $\mathcal{H}[z_1,\ldots, z_n]$,
which can be interpreted as
new partial derivatives  for a convergent power series as in (\ref{f})
 with respect to each  of the  variables $z_1,\ldots, z_n$.

\begin{definition}
If $f$ is a convergent power series of variables $z_1,\ldots, z_n$,
for a given $j\in\mathbb{N}$, $1\leq j\leq n$ and (sufficiently small)
$h\in \mathbb{H}$ , we say that $\widehat{\partial}_{z_j}
f(z_1,\ldots, z_n)[h]$ is to be defined by the position

\[
  f(z_1, \ldots, z_j+h,\ldots, z_n) - f(z_1,\ldots, z_j, \ldots, z_n)
= \widehat{\partial}_{z_j}f(z_1,\ldots, z_n)[h] + o(\|h\|),
\]
or equivalently
\[
  \widehat{\partial}_{z_j}f(z_1,\ldots, z_n)[h]=\lim_{t \ra 0}
  \frac{1}{t}\left(f(z_1, \ldots, z_j+th,\ldots, z_n) - f(z_1,\ldots,
  z_j, \ldots, z_n)\right).
\]

\end{definition}

\noindent All the operators $\widehat{\partial}_{z_j}$
are additive and right--$\mathbb{H}$--linear.
Furthemore, the Leibniz rule holds.

In practice, each of the operators $\widehat{\partial}_{z_j}$
 acts by replacing a prescribed variable in each monomial of $f_d$
 with $h \in \mathbb{H}$ as in the following
example
\[
  \widehat{\partial}_{z_1}(z_1z_2z_1^2 z_2 a)[h]:= (h z_2 z_1^2 z_2  + z_1 z_2 h z_1 z_2  + z_1 z_2 z_1 h z_2 ) a.
\]

The following result, whose proof is somehow redundant,  motivates the introduction of
the differential operators $\widehat{\partial}_{z_j}$
on $\mathcal{H}[z_1,\ldots, z_n]$.
\begin{lemma}\label{lem1}
  If $\widehat{\partial}_{z_j}f(z_1,\ldots, z_n)\equiv 0$, then $f(z_1, \ldots, z_n)$ is (formally) independent of $z_j$.
\end{lemma}

\begin{remark}
{\em One can also define the
(differential) operator
\begin{equation}\label{df2}
\widetilde{\partial}_{z_j} f(z_1,\ldots, z_n) :=\widehat{\partial}_{z_j} f(z_1,\ldots, z_n)[1],
\end{equation}
which coincides with the corresponding (Cullen) derivative, when $f$ is a
slice-regular functon.
In short, the operator
$\widetilde{\partial}_{z_j}$
replaces each  $z_j$  with $1$.

However, a result like the one in Lemma \ref{lem1} doesn't hold when considering
$\widetilde{\partial}$ instead of $\widehat{\partial}.$ Indeed,
\[ \widetilde{\partial}_{z_1}(z_1z_2-z_2z_1)=0  \]
but the function $f(z_1,z_2)=z_1z_2-z_2z_1$ does not depend on $z_2$ only.}
\end{remark}

\subsection{Derivatives of mappings}

Consider a mapping $F = (f_1, f_2),\, f_1,f_2 \in {\mathcal H}[z,w]$  and  define
$$
  DF(z,w)[h_1,h_2]:= \begin{bmatrix}
   \whp_z f_1(z,w)[h_1] & \whp_w f_1(z,w)[h_2]  \\
   \whp_z f_2(z,w)[h_1] & \whp_w f_2(z,w)[h_2]
          \end{bmatrix}
$$
  Let $G=(g_1,g_2),\, g_1, g_2 \in {\mathcal H}[z,w],$ and write $(u,v) = G(z,w).$
 If

$$
  DG(z,w)[h_1,h_2] = \begin{bmatrix}
   \whp_z g_1(z,w)[h_1] & \whp_w g_1(z,w)[h_2]  \\
   \whp_z g_2(z,w)[h_1] & \whp_w g_2(z,w)[h_2]
   \end{bmatrix} =
   \begin{bmatrix}
     a_1 & b_1\\
     a_2 & b_2
   \end{bmatrix}
$$
then we define the derivative of the composition as

\[
D(F \circ G)(z,w)[h_1,h_2] =\begin{bmatrix}
   \whp_z f_1(u,v)[a_1] + \whp_w f_1(u,v)[a_2]  & \whp_z f_1(u,v)[b_1] + \whp_w f_1(u,v)[b_2] \\
   \whp_z f_2(u,v)[a_1] + \whp_w f_2(u,v)[a_2]  &\whp_z f_2(u,v)[b_1] + \whp_w f_2(u,v)[b_2]
  \end{bmatrix}.\]
We introduce a  new notation and write
\begin{equation}\label{diamond}
 D(F \circ G)(z,w)[h_1,h_2] =
\begin{bmatrix}
   \whp_z f_1(u,v) & \whp_w f_1(u,v)  \\
   \whp_z f_2(u,v) & \whp_w f_2(u,v)
  \end{bmatrix}
  \diamond
    \begin{bmatrix}
     a_1 & b_1\\
     a_2 & b_2
   \end{bmatrix}
\end{equation}
so that
\[ D(F \circ G)(z,w)[h_1,h_2]  =  DF(G(z,w))\diamond DG(z,w)[h_1,h_2].
\]

\subsection{Bidegree full functions (in two variables)}

Even though many of the following considerations can be given
in a general formulation for $f\in \mathcal{H}[z_1,\ldots, z_n]$,
for the sake of clearness and to avoid complicated notations,
we'll focus our attention to the two variable case and denote $z_1=z$,   $z_2=w.$

In each $\mathbb{H}_d[z,w],$
we consider the submodule $\mathbb{H}^1_d[z,w]$ whose elements are finite linear combinations of
monomials 
of {\em bidegree} $(p, q), p,q \geq 0, p + q = d$ with respect to the
variables $z$ and $w$; they are all the monomials of total
degree $d$ formed considering $p$ copies of $z$'s and $q$ copies of
$w$'s.
There are $\binom{p+q}{p}$ such monomials and each of them can
be represented by a string (called a {\em word})  $\alpha^{p,q}
=(\alpha_1^{p,q},\ldots,\alpha^{p,q}_{d})\in \{0,1\}^d$ such that
$|\alpha^{p,q}| := \sum\limits_{l=1}^{d}
\alpha_l^{p,q} = p.$ With this notation we can write
\[(z,w)^{\alpha^{p,q}} :=
(z^{\alpha^{p,q}_1}w^{1-\alpha^{p,q}_1}) \cdot \ldots \cdot
(z^{\alpha^{p,q}_{d}}w^{1-\alpha^{p,q}_{d}}).\]

\noindent Notice that, if
$f(z,w) = \sum\limits_d f_d(z,w)\in\mathcal{H}^1[z,w]$,  then

\[f_d(z,w)=\sum_{\alpha^{p,q}} (z,w)^{\alpha^{p,q}} a_{\alpha^{p,q}}\]
with $p+q=d$.

\noindent Denote by

\[S_{p,q}(z,w):=\sum\limits_
{\substack {\alpha^{p,q},\\ |\alpha^{p,q}| = p\\ p+q=d}
}(z,w)^{\alpha^{p,q} }.\] It is clear that $S_{p,q}(z,w)=
S_{q,p}(w,z).$ 
If $z$ and $w$ commute, then
$S_{p,q}(z,w) = \binom{p+q}{p} z^pw^q$. We also have this important
identity
\begin{equation}\label{e4}
\widehat{\partial}_z S_{p+1,q}(z,w)[h] = \widehat{\partial}_w S_{p,q+1}(z,w)[h].
\end{equation}

Proving that monomials of bidegree $(p,q)$ are not just formally
(right) linearly independent, but (right) linearly independent as
functions, is a nontrivial problem. However, we can prove this fact for some
cases.
\begin{proposition}\label{linearindip}
Consider a polynomial of bidegree $(p,q)$ with $p+q=d$ and either $q \leq 1$ or $p \leq 1,$
 \[P_{p,q} (z,w) =
\sum\limits_
{\substack
{\alpha^{p,q},\\
|\alpha^{p,q}| =  p} }(z,w)^{\alpha^{p,q} }a_{\alpha^{p,q} };\].
If $P_{p,q} (z,w)\equiv 0$ then necessarily
$a_{\alpha^{p,q} }=0$ for any $\alpha^{p,q}$.
\end{proposition}

\begin{proof}
The cases $p=0$ or $q=0$ are trivial.
If $q=1$ then we can use a simpler notation and write

 \[P_{p,1} (z,w) = \sum\limits_ {n=0}^d z^nwz^{d-n}a_n.\]
 If  $P_{p,1} (z,w)\equiv 0$, then in particular
 $P_{p,1} (z,w) = 0$ for $z=x+Iy$  and $w=J\in\mathbb{S}$ an imaginary unit orthogonal to $I$ such that
 $\{I,J,IJ\}$ is an orthonormal basis of $\mathbb{R}^3$.
 In particular, this choice of $J$ implies that  $zw=w\bar z$.
 Hence
 \[0=P_{p,1} (z,w) = w\sum\limits_ {n=0}^d \bar{z}^nz^{d-n}a_n;\]
since $w=J\neq 0$, it follows that

\[\sum\limits_ {n=0}^d \bar{z}^nz^{d-n}a_n\equiv 0\]
for any choice of $x,y\in\mathbb{R}$ or  $z\in\mathbb{C}_{I}:=\{z=x+Iy\ |\ x,y\in\mathbb{R} \}\simeq\mathbb{C} $.
Since for any $n$ it turns out that $a_n=u_n+v_nJ$ with $u_n,v_n\in\mathbb{C}_I$, then
$\sum\limits_ {n=0}^d \bar{z}^nz^{d-n}a_n=0$
splits into two independent conditions (on $\mathbb{C}_I$), namely
$\sum\limits_ {n=0}^d \bar{z}^nz^{d-n}u_n=0$ and
$\sum\limits_ {n=0}^d \bar{z}^nz^{d-n}v_n=0;$
from the Identity Principle for complex polynomials, we conclude, that $u_n=0$ and $v_n=0$ for any $n$ and so
$a_n=0$ for $n=0,\ldots , d$.
\end{proof}

\begin{definition}

We define
\[ \mathbb{H}^{BF}_d[z,w]:=
\left\{ \sum_{p+q=d}
  S_{p,q}(z,w)a_{p,q},\: a_{p,q} \in \mathbb{H}\right\}\]
and
\[ \mathbb{H}^{BF}[z,w]
:=\bigoplus\limits_d  \mathbb{H}^{BF}_d[z, w].\]


We say that $\mathbb{H}^{BF}[z,w]$ is the right module of
 {\em bidegree  full} (in short BF)   polynomials in the variables $z,w$.
Similarly, we define the right module of {\em  bidegree full functions} to consist of
converging  power series of the form
\[
  f(z,w) = 
\sum_{d = 0}^{\infty} f_d(z,w),
\]
with $f_d(z,w)\in  \mathbb{H}^{BF}_d[z,w]$ and denote it by $\mathcal{H}^{BF}[z,w]$.
\end{definition}

The following result shows that bidegree full polynomials form an
interesting class of polynomials.
\begin{lemma}\label{l1}

For any real number $\mu$ and any $d\in\mathbb{N}$, the polynomial
$(z-\mu w)^d:= \overbrace{(z - \mu w)\cdots (z - \mu w)}^{d\ \mathrm{ times}}$
 is
bidegree full.
If $P(z,w)= \sum\limits_{d = 0}^l \sum\limits_{\substack{p,q \geq 0,\\
p+q = d}}
  S_{p,q}(z,w)a_{p,q}$ is a bidegree full polynomial
of degree $d$, then it also has a decomposition
\begin{equation}\label{decomp}
  P(z,w) = \sum\limits_{d = 0}^l \sum\limits_{ p+q=d}
  \left(\sum_{n = 0}^d(z - n w)^d r_{p,d}(n)\right)a_{p,q}, \; {with}\ r_{p,d}(n)\in \mathbb{R}.
\end{equation}
\end{lemma}
\begin{proof}

Indeed, from direct calculations, it follows that
\[
 (z-\mu w)^d =(z-\mu w)\cdot \ldots \cdot (z - \mu w) = \sum_{\substack{p,q \geq 0,\ \\ p + q = d}}  S_{p,q}(z,w)(-\mu)^{q}.
\]
The second statement follows from the fact
(proved in  \cite{A} by induction on $d$ with an argument which
applies to our setting) that the polynomials
$\{x^d,(x-1)^d,\ldots, (x-d)^d\}$
form a basis of real polynomials of order less or equal to $d$
and consequently polynomials
$z^d,(z-w)^d,\ldots, (z-dw)^d$  form a basis of $ \mathbb{H}^{BF}_d[z,w]$


\end{proof}

 Notice, furthermore,  that
\begin{equation}\label{e3}
\widehat{\partial}_w (z - \mu w)^d =-\mu \widehat{\partial}_z (z - \mu w)^d
\end{equation}
if and only if $\mu \in \Rr$.


\begin{remark} {\em
As a consequence of Lemma \ref{l1}, from any convergent quaternionic power series
in the variable $u$ of the form
\[u\mapsto \sum_d u^d a_d\]
(which actually is a slice--regular function of $u$)
one gets a bidegree full function by replacing $u$ with $z-\mu w$, namely
\[f(z,w)=\sum_d (z- \mu w)^d a_d \in\mathcal{H}^{BF}[z,w];\]
this function is not a slice--regular
function in the variables $z$ and $w$.}

\end{remark}

\subsection{Generalizations of bidegree full functions}

The generators $z^d,(z-w)^d,\ldots, (z-dw)^d$ of $
\mathbb{H}^{BF}_d[z,w]$ were obtained by precomposing the monomial
$u^d$ by functions $u = z - nw$ for $n = 0,\ldots, d.$

Similarly, given $a =(a_1,\ldots, a_d),$ one can consider the monomial
of degree $d$ in variable $u$ of the form
$$
   a_0 u a_1 u \ldots a_{d} u.
$$ Precomposing it by functions $u = z - nw$ for $n = 0,\ldots, d,$
   one obtains generators of the right module of generalized BF
   polynomials of degree $d$ denoted by $\mathbb{H}^{BF,a}_d[z,w].$

Another possible generalization is to consider the precompositions of
the slice--regular functions $f(u) = \sum_d u^d a_d$ by $u = z - \mu w$
as in Lemma \ref{l1} with $\mu \in \mathbb{H},$

\[ f(z - \mu w) = \sum_d (z-\mu w)^na_n, \quad a_n\in\mathbb{H}.\]
These functions have the geometric property of leaving invariant
quaternionic parallel affine subsets along the direction $(\mu,1)$ as
explained in the next

\begin{definition}
  Given $\mu\in\mathbb{H}$, we say that a quaternionic function $f$ of the variables $z,w$ is
$(\mu,1)$--right-invariant if
\[f(z,w)=f(z+\mu s, w+s)=f((z,w)+(\mu,1)s)=f(z -\mu w)\]
for any $z,w$ and any $s\in\mathbb{H}$.
\end{definition}



\section{Quaternionic Vector Fields in two variables}

In this section, using the definition of $\hat{\partial}$, we develop some analytic tools such as divergence,
rotor, and flow for quaternionic vector fields in two variables.  We
show that there is a large class of vector fields with good
analyticity properties.

\begin{definition}
Given $f,g\in \mathcal{H}[z,w]$,  the mapping
$X(z,w)=(f(z,w), g(z,w))$ is called a
 {\em  vector field} in $\mathbb{H}^2,$
 in short we write $X \in \mathcal{VH}.$ The subset of vector fields $X = (f,g)$  with $f,g \in \mathcal{H}^1[z,w]$ is denoted by
$\mathcal{VH}^1$.
In particular, we say, that a vector field
$X =(f,g)$ is {\em bidegree full} (in
short BF) if the functions $f,g$ are bidegree full functions and use the notation
$X \in \mathcal{VH}^{BF}$.
We assume from now on that the vector fields and functions are all defined on $\mathbb{H}^2$.

\end{definition}
Next we introduce the following

\begin{definition}
Given the vector field $X(z,w)=(f(z,w), g(z,w))$,
we define the differential) operator
\[
  \mathrm{Div}X(z,w)[h] := \widehat{\partial}_zf(z,w)[h] + \widehat{\partial}_wg(z,w)[h]
\]
and we say that the vector field $X$ has  {\em divergence} if
$ \mathrm{Div}X(z,w)[h]$ is left $h$ linear, i.e. if
there exists a function -- which will be  denoted by $\mathrm{div}X(z,w)$ -- such that

\[
\mathrm{Div}X(z,w)[h] = h\, \mathrm{div} X(z,w).
\]

\end{definition}

\example The vector field $(zw + wz, -w^2)$ has divergence zero,
\[
  \mathrm{Div}(z w + w z, -w^2)[h] = h  w + w h  - (h  w + w h)= 0,
\]
while the vector field $( z^2 w, -z w^2)$ does not have  divergence, since the operator
\[
  \mathrm{Div}(z^2 w, -z w^2)[h] = (h z  + zh) w   - z (h  w + w h) = hzw-zwh
\]
is not left linear in $h.$

\medskip
One of the main reasons for the introduction of the
operators $\widehat{\partial}_z, \widehat{\partial}_w$ and  $\mathrm{Div}$ is the following

\begin{theorem}\label{crcthm} Let $X(z,w) = (f(z,w),g(z,w)) \in \mathcal{VH}^1$ be a vector field with  divergence. Then
$\mathrm{div} X(z,w)$ is BF. If $\mathrm{div} X(z,w) = 0$ then $X$ is BF.
\end{theorem}

\begin{proof}
 To simplify the notation
 write $\mathrm{div}X(z,w) = \Delta(z,w).$
Let $f(z,w) = \sum f_{p,q}(z,w),$ $ g(z,w) = \sum g_{p,q}(z,w)$ and $\Delta(z,w) = \sum \Delta_{p,q}(z,w)$ be
the decompositions of $f,$  $g$ and $\Delta$  with respect to the bidegrees.  Then
$\Div X(z,w)[h]= h  \Delta(z,w),$ iff
\begin{equation}\label{e5}
   \widehat{\partial}_z f_{p+1,q}(z,w)[h] +   \widehat{\partial}_w g_{p,q+1}(z,w)[h] = h \Delta_{p,q}(z,w)
\end{equation}
for $p,q \geq 0.$ We have two more equations, which always hold, namely,
$$
  \whp_z f_{0,q}(z,w)[h] = 0\mbox{  and  }\whp_w g_{p,0}(z,w)[h] = 0.
$$
Write
\begin{eqnarray*}
  f_{p+1,q}(z,w) &=& z\sum\limits_{\substack{\alpha_1 \in \lbrace0,1\rbrace^{p+q},\\|\alpha_1| = p}}  (z,w)^{\alpha_1} A_{\alpha_1}
             + w\sum\limits_{\substack{\alpha_2 \in \lbrace0,1\rbrace^{p+q},\\|\alpha_2| = p+1}}  (z,w)^{\alpha_2} A_{\alpha_2},\\
  g_{p,q+1}(z,w) &=& z\sum\limits_{\substack{\beta_1 \in \lbrace0,1\rbrace^{p+q},\\|\beta_1| = p-1}}(z,w)^{\beta_1}B_{\beta_1} +
                w\sum\limits_{\substack{\beta_2 \in \lbrace0,1\rbrace^{p+q},\\|\beta_2| = p}}(z,w)^{\beta_2}B_{\beta_2}.\\
\end{eqnarray*}
Since divergence is left linear in $h$, all the terms in the
derivative coming from the second sum for $f_{p+1,q}$ (similarly for
the first sum for $g_{p,q+1}$) should cancel out. Since the terms in
the expression $\widehat{\partial}_zf_{p+1,q}(z,w)[h]$ are formally linearly
independent, the only possibility is, that such a term is cancelled out
by a term in $\widehat{\partial}_zg_{p,q+1}(z,w)[h].$ Consider a
monomial from the second sum whose associated word is
of the form $0\alpha_2 = 0 \alpha_2^1 1
\alpha_2^2 0 \alpha_2^3.$ Then
$\widehat{\partial_z}(z,w)^{0\alpha_2}[h]$ has a monomial of the form
$w \cdot(z,w)^{\alpha_2^1} \cdot h \cdot (z,w)^{\alpha_2^2} \cdot w \cdot (z,w)^{\alpha_2^3},$ so
it can be cancelled out only by a term in
$\widehat{\partial}_w(z,w)^{\beta}[h]$ for $\beta = 0 \alpha_2^1 0
\alpha_2^2 0 \alpha_2^3.$ Since there is another zero, the above
derivative contains also a term $w \cdot (z,w)^{\alpha_2^1}\cdot w \cdot
(z,w)^{\alpha_2^2}\cdot h \cdot (z,w)^{\alpha_2^3},$ and this one can be
cancelled only by a term from $\widehat{\partial_z}(z,w)^{\alpha}[h]$
for $\alpha = 0 \alpha_2^1 0 \alpha_2^2 1 \alpha_2^3 = 0
\tilde{\alpha}_2.$ The sequences $\alpha_2$ and $\tilde{\alpha}_2$
differ only by a transposition. So, if both $\alpha_2$ and
$\tilde{\alpha}_2$ with $|\alpha| = |\alpha_2| = p+1$ contain at
least one $1$ (which is the case) and one $0$, they differ by a
sequence of transpositions and therefore $A_{\alpha_2} =
A_{\tilde{\alpha}_2}.$ So, there exist $A$ such that
 \[ A = A_{\alpha_2} = A_{\tilde{\alpha}_2} = -B_{\beta_2}, \quad \forall \alpha_2, \beta_2,
 \]
 provided $q \geq 2$ (and $p+1 \geq 1$).
 Analogously, there  exist $B$ such that
\[ B = - B_{\beta_1} = A_{1 \alpha_1}, \quad \forall \alpha_1, \beta_1
 \]
 if $q \geq 1$ and $p \geq 1.$
 Then
 \begin{eqnarray*}
   (f_{p+1,q}(z,w), g_{p, q+1}(z,w))& =&\\
   &&\hspace{-4cm}= (z S_{p,q}(z,w) B + w S_{p+1,q-1}(z,w) A, - z S_{p-1, q+1}(z,w) B - w S_{p,q}(z,w) A)\\
   &&\hspace{-4cm} =z(S_{p,q}(z,w),-  S_{p-1, q+1}(z,w))B + w(  S_{p+1,q-1}(z,w), -  S_{p,q}(z,w)) A,
 \end{eqnarray*}
 and
 \begin{eqnarray*}
   h \Delta_{p,q}(z,w) &=& h S_{p,q}(z,w) B + z \whp_z S_{p,q}(z,w) [h]B + w \whp_z S_{p+1,q-1}(z,w) [h]A - \\
       &&\quad - z \whp_w S_{p-1,q+1}(z,w) [h] B - h S_{p,q}(z,w) [h] A - w \whp_z S_{p,q}(z,w) [h]A =\\
       &=& h S_{p,q}(z,w) (B - A),
 \end{eqnarray*}
 since by (\ref{e4}) we have
 \[
    \whp_z S_{p,q}(z,w) [h] =  \whp_w S_{p-1,q+1}(z,w)[h], \;  \whp_z S_{p+1,q-1}(z,w) [h] =  \whp_w S_{p,q}(z,w)[h], \;
 \]
 thus $\Delta_{p,q}$ is BF and $\div (S_{p,q}(z,w),-  S_{p-1, q+1}(z,w)) = 0$ for all $p \geq 1, q\geq 0.$
 If divergence is $0,$ then also $A = B$ and
 \[
    (f_{p+1,q}(z,w), g_{p, q+1}(z,w)) = (S_{p+1,q}(z,w), -  S_{p,q+1}(z,w)) A.
 \]
 We have three remaining cases to check separately, $p =0, q=0$ and $q=1.$  In the first case,
 we have a degree $q+1$ vector field $X(z,w) = (f_{1,q}(z,w), g_{0,q+1}(z,w))$,
 $$ f_{1,q} (z,w)=z w^{q} A_{q} + w \sum_{\alpha \in
   \lbrace0,1\rbrace^{q-1},|\alpha| = 1} (z,w)^{\alpha}
 A_{\alpha},\, g_{0,q+1} = w^{q+1} B.\\
 $$ Since there is only one element in the second component, it
   follows that $B = -A_{\alpha}$ for all $\alpha$ and so the vector
   field is of the form
$$
  (z w^q A_q - w S_{1,q -1}(z,w)B,  w^{q+1} B) = (z w^q ,0) A_q + (-w S_{1,q-1}(z,w), S_{0,q+1}) B.
$$ with divergence equal to $w^q (A_q + B).$  Again, if divergence is $0,$
   then $A_q = -B$ and the vector field is of the form
\[
  (z w^q A_q - w S_{1,q -1}(z,w)B,  w^{q+1} B) = (- S_{1,q}(z,w), S_{0,q+1}) B.
\]
The second is the case of vector
   fields of the form $X(z,w) = (f_{p+1,0}(z,w), g_{p,1}(z,w))$ and is treated similarly as the first case.
   In the third case we have vector fields of the form  $X(z,w) = (f_{p+1,1}(z,w), g_{p,2}(z,w))$  and because the case $p = 0$ is already proved we assume $p > 0.$ Then  there is only one $A_{\alpha_2} = A$ and so $B_{\beta_2} + A = 0,$ therefore the vector fields are of the form
   \begin{eqnarray*}
     (f_{p+1,1}, g_{p,2})(z,w) &=& z\left( \sum\limits_{\substack{\alpha_1 \in \lbrace0,1\rbrace^{p+2},\\|\alpha_1| = p}}  (z,w)^{\alpha_1} A_{\alpha_1}, \sum\limits_{\substack{\beta_1 \in \lbrace0,1\rbrace^{p+2},\\|\beta_1| = p-1}}  (z,w)^{\beta_1} A_{\beta _1}\right) + \\
     &&+(w z^p, - w S_{p,1}(z,w))A.
   \end{eqnarray*}
   Since there are two zeroes in $\beta_1$ and one zero in $\alpha_1,$ we can apply the same transposition argument as above, but to the word of the form $1\alpha_1 = 1 \alpha_1^1 1 \alpha_1^2 0 \alpha_1^3$ and conclude, that for any two words $\alpha_1,$ $\alpha_2$  we have  $A_{\alpha_1} = A_{\alpha_2} = -B_{\beta_1}  = B,$ so
   $ (f_{p+1,1}, g_{p,2})(z,w) = z(S_{p,1}(z,w), - S_{p-1,2}(z,w))B + w( S_{p,0}(z,w), -  S_{p,1}(z,w))A$ with divergence equal to
   $$
     \mathrm{div} (f_{p+1,1}, g_{p,2})(z,w) = S_{p,1}(z,w)(B - A).
   $$
   If divergence is $0,$ then the vector field is of the form
   $ (f_{p+1,1}, g_{p,2})(z,w) = (z S_{p,1}(z,w) + w S_{p,0}(z,w), - z S_{p-1,2}(z,w)-w S_{p,1}(z,w))A =(S_{p+1,1}(z,w), - S_{p,2}(z,w))A, $ so it is BF.
\end{proof}
\noindent An immediate consequence of the proof is the following
\begin{corollary}\label{c1}
  Let $X(z,w) \in \mathcal{VH}^1$  be a vector field with divergence. Then it has a form
\begin{eqnarray*}
    X(z,w) &=& (z \sum_{p \geq 1} (S_{p,q}(z,w),- S_{p-1, q+1}(z,w))a_{p,q} +\\
            &&  w \sum_{q \geq 1}(S_{p+1,q-1}(z,w), - S_{p,q}(z,w)) b_{p,q}) +
              (g_0(w), f_0(z))
\end{eqnarray*}
and its divergence is $\div X(z,w) =\sum_{p,q \geq 0} S_{p,q}(z,w)(a_{p,q} - b_{p,q}).$
\end{corollary}

\begin{definition}
Given the vector field $X(z,w)=(f(z,w), g(z,w))$,
we define the differential) operator
\[
  \mathrm{Rot}X(z,w)[h] := -\widehat{\partial}_zg(z,w)[h] + \widehat{\partial}_wf(z,w)[h]
\]
and we say that the vector field $X$ has {\em rotor} if
$ \mathrm{Rot}X(z,w)[h]$ is left $h$ linear, in other words if
there exists a function -- which will be  denoted by $\mathrm{rot}X(z,w)$ -- such that

\[
\mathrm{Rot}X(z,w)[h] = h\, \mathrm{rot} X(z,w).
\]

\end{definition}

Since $\mathrm{Rot}(f,g) = \mathrm{Div}(-g,f),$ we immediately have the following

\begin{theorem} Let $X(z,w) = (f(z,w),g(z,w)) \in \mathcal{VH}^1$ be a vector field with  rotor. Then
$\mathrm{rot} X(z,w)$ is BF. If $\mathrm{rot} X(z,w) = 0,$ then $X$ is BF and has the form
$$
    X(z,w) = \sum_{p,q \geq 1} (S_{p-1,q}(z,w), S_{p, q-1}(z,w))a_{p,q} +
              (\sum_{p \geq 0} z^p a_p,\sum_{q \geq 0} w^q b_q).
$$
  Define
$$ \chi(z,w):= \sum_{p,q \geq 1} S_{p,q}(z,w)\frac{a_{p,q}}{p+q} +
  \sum_{p \geq 0} \left( z^{p+1} \frac{a_p}{p+1}+ w^{p+1}
  \frac{b_p}{p+1}\right) + C,
$$
 where $C \in \mathbb{H}$ is an arbitrary constant. Then
  \[
    X(z,w) = ({\widetilde\partial}_z\chi(z,w), {\widetilde\partial}_w \chi(z,w)).
  \]
\end{theorem}
\begin{proof}
By definition (\ref{df2}) of derivatives ${\widetilde\partial}_z$ and ${\widetilde\partial}_w$   we have
$$ {\widetilde\partial}_z S_{p,q}(z,w) = (p+q)S_{p-1,q}(z,w) \mbox{
  and } {\widetilde\partial}_w S_{p,q}(z,w) = (p+q)S_{p,q-1}(z,w).
$$
\end{proof}

\begin{definition}
  Let $D \subset \mathbb{H}^2 \times \Rr$ be an open set containing $\mathbb{H}^2 \times\{0\}.$ A function
  $$
    \Phi^X:  D \rightarrow \mathbb{H}^2
  $$
  is a {\em flow} of the vector field $X$ if
  $$
    \frac{d}{dt}\Phi^X(z,w,t) = X(\Phi^X(z,w,t)), \quad \forall (z,w,t) \in D.
  $$
and
\[\Phi^X(z,w,0) = (z,w), \quad \forall (z,w) \in  \mathbb{H}^2.\]

  If $D = \mathbb{H}^2 \times \Rr,$ we say that a vector field $X$ is complete.
\end{definition}
Whenever it is clear from the context which vector field we are
referring to, we  omit the superscript $X$.

\example\label{SHOS}
 Consider the vector fields
$$
  X(z,w) = (f(w), 0) \mbox{ and } Y(z,w) = (z g(w),0)
$$
with $f$ and $g$ slice--regular functions
 defined on $\mathbb{H}.$
We have
$$
  \div X(z,w) = 0 \mbox{ and } \div Y(z,w) = g(w).
$$
The corresponding flows are
\begin{equation}\label{e7}
  \Phi^X(z,w,t) = (z,w) + t (f(w),0) \mbox{ and } \Phi^Y(z,w,t)= (z,w) + (z (e^{t g(w)}-1),0)
\end{equation}
and the vector fields are complete. The exponential function is defined by series expansion,
$$
  e^{t g(w)} = \sum_{0}^{\infty}\frac{t^n g(w)^n}{n!}
$$
and is not a slice--regular function in general.

\example \label{TheExample} The vector field $X(z,w) = (z^2w, - z
w^2)$ is complete with a flow $\Phi^X(z,w,t) = (z e^{t z w}, e^{-tzw}
w) = (u,v):$ 
\begin{eqnarray*}
  \frac{d}{dt}(z e^{t z w}, e^{-tzw} w)&=&(z e^{t z w}zw, -zw e^{-tzw} w)\\
  &=& ((z e^{t z w})(z e^{t z w})(e^{-t z w} w), -(ze^{t z w})(e^{-t z w}w)( e^{-tzw} w))\\
  &=& (u^2v, - u v^2).
\end{eqnarray*}
 Because $\Div X(z,w)[h] = hzw-zwh,$ the vector field $X$ does not have  divergence.

\section{Quaternionic Determinants and applications to Vector Fields of Shear and Overshear Automprphisms}

This chapter is mainly devoted to the study
of special classes of vector fields which are generalizations
of the two
vector fields from  example (\ref{SHOS}).
We  focus, in particular,  on  the geometric  properties of
the divergence  of the  flows of these vector fields.

If $A$ is an invertible real matrix
$$
   \begin{bmatrix}
        a&b\\
        c&d
       \end{bmatrix}\in GL(n,\mathbb{R})
  $$
and $f  \in {\mathcal H}(\mathbb{H})$,
we consider
the vector field
$$
  X(z,w) = \frac{1}{ad - bc}(d,-c) f(cz + dw).
  $$
  If $\pi_2:\mathbb{H}^2\to \mathbb{H}$ is the projection onto the second coordinate, one can write
  $X(z,w) = A^{-1} (f \circ \pi_2,0)^T ( A \cdot (z,w)^T)$
  Notice that if $d = 0,$ the vector field is of the form $(0,g(z))$
  and if $c = 0$ is of the form $(g(w),0)$ for a suitable $g \in {\mathcal H}(\mathbb{H})$.
  In both cases, the vector field $X$  has divergence $0.$

  Assume now that $c \ne 0$. Then
\begin{eqnarray*}
  \Div X(z,w)[h] &=& \frac{1}{ad - bc} (\whp_z f(cz + dw)[h]  d +\whp_w f(cz + dw)[h] (-c) \\
             &=& \frac{1}{ad - bc} (\whp_z f(cz + dw)[h] d +\whp_z f(cz + dw)[h]c^{-1} d (-c)=0 \\
\end{eqnarray*}
If $c \ne 0,$  we may assume that $c = -1$.
If we write $d = \mu$,  the vector field $X$ can be written in a form
$$
  X(z,w) = (\mu,1)\tilde{f}(z - \mu w)
$$
for some other slice-regular function $\tilde {f}.$
Notice that the vector field $X$ is in the kernel of the functional $\Lambda (z,w) = z - \mu w,$ i.e. $\Lambda(X) = 0.$

If $\pi_1:\mathbb{H}^2\to \mathbb{H}$ is the projection onto the first coordinate, consider the vector field
\begin{equation}\label{osvf}
 Y(z,w) = A^{-1}  (\pi_1 \cdot f \circ \pi_2,0)^T ( A\cdot (z,w)^T)= \frac{1}{ad - bc}(d,-c) (az + b w) f(cz + dw).
\end{equation}
It has divergence
\begin{eqnarray*}
  \Div Y(z,w)[h] &=& \frac{1}{ad - bc}\left[(az + bw) (\whp_z f(cz + dw)[h]  d +\whp_w f(cz + dw)[h] (-c))\right. \\
      && \left.+ (a d - b c) h f(cz + dw) \right] = h f(cz + dw).
\end{eqnarray*}
Similarly s before, $\Lambda(Y) = 0$ for $\Lambda(z,w) = z - \mu w.$

\begin{definition}
Let $\pi_1, \pi_2$ denote the projections of $\mathbb{H}^2$ on the
first and second coordinate respectively.  We define the following
two  classes of vector fields:
$$\begin{array}{rcl}
  SV_{\mathbb{R}} &=& \lbrace X, \, X(z,w) = A^{-1} (f \circ
    \pi_2,0)^T ( A \cdot (z,w)^T),\; A \in SL(2,\mathbb{R}), f \in \mathcal {S}\mathcal{ R}(\mathbb{H}) \rbrace,\\
 OV_{\mathbb{R}} &=& \lbrace Y, \,
       Y(z,w) = A^{-1} (\pi_1 \cdot f \circ \pi_2,0)^T ( A\cdot
       (z,w)^T),\; A \in GL(2,\mathbb{R}), f \in {\mathcal
         H}(\mathbb{H}) \rbrace.
\end{array}$$
The classes $SV_{\mathbb{R}}$ and $OV_{\mathbb{R}}$ are called {\em
  shear} and {\em overshear vector fields} respectively.

\end{definition}

The space of all shears $SV_{\mathbb{R}}$ can also be described as
$$
  SV_{\mathbb{R}} =  \lbrace (r,1) f(z - rw), \; r \in \mathbb{R},
f \in {\mathcal SR}(\mathbb{H}) \rbrace\cup
\lbrace (g(w),0)\;  g\in {\mathcal SR}(\mathbb{H}) \rbrace
$$

\begin{lemma}\label{diver}  For each $p,q$ there exists a vector field $Y_{p,q}$ with $\div Y_{p,q}(z,w) = S_{p,q}(z,w)$ and it is a sum of overshear vector fields.
\end{lemma}
\begin{proof}
Since $S_{p,q}(z,w) = \sum\limits_{n=0}^{p+q} (z - nw)^{p+q} r_{n}, r_n \in \mathbb{R}$ by formula (\ref{decomp}), the vector field is $$Y_{p,q}(z,w) =\sum_{n=0}^{p+q} (n,1)(z + nw)(z - nw)^{p+q}
\frac{-r_{n}}{n^2 + 1}.$$
\end{proof}

\begin{proposition}\label{p1}
Any polynomial vector field $X \in {\mathcal VH}^1$ with divergence is a
finite sum of shear  and overshear
 vector fields. If $\mathrm{div} X =0,$ then $X$ can
be written as a sum of shear vector fields.
\end{proposition}

\begin{proof} Let $X = \sum_d X_d$ be the homogenous expansion of a vector field $X$.
   Since divergence of $X$ is bidegree full, by Lemma \ref{diver} there exists a vector field $Y$, which is a sum of overshear vector fields,  such that $\div X = \div Y,$ so it is sufficient
to prove that every divergence zero vector field is a sum of shear
vector fields.  Since the operator $\Div$ respects the degree in the expansion,  it suffices to prove the assertion for
each fixed degree.
Now assume that $\div X_d = 0.$ Because of Lemma
\ref{l1}, we can write $X_d$ as
\[
  X_d(z,w) = \left(\sum_{n = 0}^d (z - n w)^da_{n,d},\sum_{n = 0}^d (z - nw)^db_{n,d}\right).
  \]
  Therefore
\begin{eqnarray*}
\mathrm{Div}X_d(z,w)[h] &=& \sum_{n = 0}^d \whp_z(z - n w)^{d}[h]a_{n,d} -
\sum_{n = 0}^d \whp_z(z - n w)^{d}[h] j b_{n,d}\\
&=&
\whp_z\left( \sum_{n = 0}^d (z - n w)^d(a_{n,d} -n b_{n,d}) \right)[h],
\end{eqnarray*}
so the condition $\mathrm{Div}X_d(z,w)[h]=0$ and Lemma \ref{lem1} imply
$$
  \sum_{n = 0}^d (z - n w)^d(a_{n,d} -j b_{j,d}) = w^d q
$$ for some $q \in \mathbb{H}.$ Since the monomials $(z - n w)^d, n =
  0,\ldots,d$ are generators of all BF polynomials, there exist
  constants $\lambda_0,\ldots\lambda_d$ such that
$$
  w^d =  \sum_{n = 0}^d (z - n w)^d \lambda_n.
$$ So we have $ \lambda_nq = a_{n,d} -n b_{n,d}$ and then $a_{n,d} =
  \lambda_n q+ n b_{n,d}.$
  In other words,
\begin{eqnarray*}
   X_d(z,w) &=& \left(\sum_{n=0}^d j(z - n w)^db_{n,d}  + \lambda_n q, \sum_{n =0}^d (z - n w)^d b_{n,d} \right)\\
     &=&\sum_{n=0}^{d}(n,1) (z - nw)^d b_d + (1,0) \sum_{n=0}^{d} (z - nw)^d \lambda_n q\\
     &=& \sum_{n=0}^{d}(n,1) (z - nw)^d b_d + (1,0) w^d q.
\end{eqnarray*}
As easily checked, all  vector fields in the last sum have divergence $0.$
\end{proof}

Passing from a real to a quaternionic matrix, we have to point out
that there is no canonical way to define the determinant of such a
matrix. We consider only $2\times 2$ matrices but we refer the
reader to \cite{As} and \cite{H} for further references on general
linear groups and determinants.  There are several possibilities of
introducing a generalization of the standard notion of determinant
according to the properties one is looking at. For example, the real
determinant $\det_{\mathbb{R}}$ and the complex determinant
$\det_{\mathbb{C}}$ of a quaternionic matrix are defined when a
quaternionic matrix is considered as the corresponding real or complex
matrix obtained via the identification of $\mathbb{H}$ with $
\mathbb{R}^4$ or with $ \mathbb{C}^2$ respectively-

 If
$$
  A = \begin{bmatrix}
   {{a}} & {{b}}  \\
   {c} & {d}
          \end{bmatrix},
$$ ($a,b,c,d\in\mathbb{H}$) we define
the {\em Cayley determinant} of $A$ to be
$$
  \det_C A=ad - cb.
  $$
  If $b = a$ and $c=d,$ the rank of the matrix is $1$ and the
  determinant is $ac - ca,$ which is $0$ iff $a$ and $c$
  commute. Another interesting definition is
  Dieudonn\'{e} determinant $\det_D$. The {\em Dieudonn\'{e} determinant} is defined as a
  mapping from $M(2,\mathbb{H})$ to a quotient $Q$ of the multiplicative
  subgroup $\mathbb{H}^*$ of $\mathbb{H}$ to its quotient by a
  commutator subgroup, $Q =\mathbb{H}^*/[\mathbb{H}^*,\mathbb{H}^*].$ The group $Q$
 is isomorphic to $\mathbb{R}_+,$ because the commutator
  subgroup consists precisely of all quaternionic units. For example,
  the representative of $\det_DA$ in $Q$ is defined as
$$
  \det_D A = \left\lbrace{\begin{array}{*{20}c} -cb & \text{if } a = 0 \\ ad - aca^{-1}b & \text{if } a \ne 0 \end{array}}\right. .
$$

  The quaternionic determinants $\det_D, \det_{\mathbb{R}}$ and $\det_{\mathbb{C}}$  satisfy
  the three following axioms: the determinant is $0$ if and only if the matrix
  is singular, the determinant of a product of matrices is a product
  of determinants and a particular Gaussian elimination is allowed.

  It is important to observe that the operator
$\diamond$ as in (\ref{diamond}) is not a product and therefore in general, no matter which
definition of the determinant we adopt, the determinant of a composed
mapping introduced by using $\diamond$  is not necessarily a product of determinants.


Therefore the following two groups of transformations
$$
  SL(2,\mathbb{H}),  \mbox{ and } GL(2,\mathbb{H})
$$
can be properly and correctly defined.



\begin{definition}
Let $\pi_1, \pi_2$ denote the projections of $\mathbb{H}^2$ on the
first and second coordinate respectively.  We define the following
two  classes of vector fields:
$$\begin{array}{rcl}
 SV_{\mathbb{H}} &=&       \lbrace X, \, X(z,w) = A^{-1} (f \circ \pi_2,0)^T ( A \cdot
       (z,w)^T),\; A \in SL(2,\mathbb{H}), f \in \mathcal{S}\mathcal{R}(\mathbb{H}) \rbrace,
\\ OV_{\mathbb{H}} &=& \lbrace Y, \,
       Y(z,w) = A^{-1} (\pi_1 \cdot f \circ \pi_2,0)^T ( A\cdot
       (z,w)^T),\; A \in GL(2,\mathbb{H}), f \in {\mathcal
         H}(\mathbb{H}) \rbrace.
\end{array}
$$

The classes  $SV_{\mathbb{H}}$ and $OV_{\mathbb{H}}$ are called
 {\em generalized shear} and {\em generalized overshear vector fields} respectively.
\end{definition}

\example Consider  the matrix
\begin{equation*}
  A  = \begin{bmatrix}
        \bar{\mu}&1\\
        1&-\mu
       \end{bmatrix} (1 + |\mu|^2)^{-1}, \; \mu \in \mathbb{H}.
\end{equation*}
Since the entries commute, the formula for the inverse $A^{-1}$ is the
same as in the commutative case and so the conjugation by such $A$ defines
a $OV_{\mathbb{H}}$ vector field in the same manner as in (\ref{osvf}). Unfortunately these vector fields do
not have divergence.  In fact, from the previous computation we have
\begin{eqnarray*}
  Y(z,w) &=& (\mu,1) (\bar{\mu}(1 +|\mu|^2)^{-1}z +(1 +|\mu|^2)^{-1} w) f((1 +|\mu|^2)^{-1} z -\mu (1 +|\mu|^2)^{-1} w)\\
  &=&(\mu,1)(az + bw)f(bz - aw)
\end{eqnarray*}
where $a: = \bar{\mu}(1 +|\mu|^2)^{-1} = (1 +|\mu|^2)^{-1}\bar{\mu}$ and $ b :=(1 +|\mu|^2)^{-1}.$ Notice that $\mu a  + b = 1.$
Then
\begin{eqnarray*}
  \Div Y(z,w)[h] &=& \left[\mu (a z +b w)(\whp_z f(bz - aw)[h]) +(a z +b
    w)\whp_w f(bz -a w)[h])\right. \\ && \left.+ (\mu a h + b h) f(bz -
    aw) \right] \\ &=& \left[\mu (a z +b w)(\whp_z f(bz - aw)[h]) +(a z
    +b w)\whp_w f(bz -a w)[h])\right]+\\ && h f(bz -aw),
\end{eqnarray*}
since $\mu ah + bh = h.$
The term in the brackets is not necessarily $0$ since the chain rule does not apply and $\mu$ is not real.
For example, a suitable choice of $f$ gives
$Y(z,w) = (\mu,1)(\bar{\mu}z + w)(z -\mu w)$ and then
\begin{eqnarray*}
\Div Y(z,w) [h]&= &h(1 + |\mu|^2)(z - \mu w) + \mu (\bar{\mu}z + w)(h)  - (\bar{\mu}z + w)(\mu h)\\
  &=& h(1 +|\mu|^2)(z - \mu w)^2 + \mu w h - w \mu h +|\mu|^2z h - \bar{\mu}zjh
\end{eqnarray*}
so $Y$ does not have  divergence.  Similarly, the vector field of the form
$X(z,w) =(\mu,1)f(z - \mu w)$ does not have  divergence and actually $\Div
X(z,w)[h] = \mu \whp_z f(z - \mu w) - \whp_wf(z - \mu w).$ This is $0$ if and
only if $\mu$ commutes with $w$ and $z,$ i.e. $\mu \in \Rr.$\\

\noindent The generalized shear and overshear vector fields, however, are complete. Indeed,
\begin{lemma} Let $X$ be a vector field with a (real) flow $\Phi^X.$ Let $A \in GL(2,\mathbb{H})$ and
  consider the  {\em conjugate} of $X$ i.e. $Y= A^{-1}X \circ A$. Then the flow  of $Y$ is
$$
 \Phi^Y= \Phi^{A^{-1} X \circ A} = A^{-1} \Phi^X \circ A.
$$
\end{lemma}
\begin{proof}
  Since  in the flow the time
$t$  is real, the derivation with respect to
  $t$ commutes with multiplication by a quaternionic matrix and so
  \begin{eqnarray*}
  \frac{d}{dt} A^{-1} \Phi^X \circ A &=& A^{-1} \left(\frac{d}{dt}\Phi^X\right) \circ A\\
  &=&  A^{-1} X \circ \Phi^X \circ A =  A^{-1} X \circ A \circ A^{-1} \Phi^X \circ A =\\
  &=& (A^{-1} X \circ A)\circ (A^{-1} \Phi^X\circ A),
  \end{eqnarray*}
  which proves that $A^{-1} \circ\Phi^X A$ is a flow of the vector field $A^{-1}\circ X A.$
\end{proof}
\begin{example}
\noindent The vector fields
\begin{eqnarray*}
 &&X(z,w) = (\mu,1)f(z - \mu w),\\
 &&Y(z,w) = (\mu,1)(|\mu|^2 + 1)^{-1}(\bar{\mu} z + w)f(z - \mu w)
\end{eqnarray*}
are obtained from vector fields in the example (\ref{SHOS}) by conjugation by suitable matrices, and therefore 
have the flows
\begin{eqnarray*}
   &&\Phi^X(z,w,t) = (z,w) + t(\mu ,1)f(z - \mu w), \\
   &&\Phi^Y(z,w,t) = (z,w) + (\mu ,1)(|\mu|^2 + 1)^{-1}(\bar{\mu} z + w)(e^{tf(z - \mu w)} - 1).
\end{eqnarray*}
\end{example}

\begin{definition}
Let $\Lambda: \mathbb{H}^2 \ra \mathbb{H}$ be a right $\mathbb{H}$--linear functional. Assume $v =
(v_1, v_2) \in \mathrm{ker} \Lambda,$ $\|(v_1,v_2)\| = 1.$
For  $f \in {\mathcal H},$ any  mapping of the form
$$ (z,w) \mapsto (z,w) + (v_1,v_2) f(\Lambda(z,w))
$$
is called a {\em generalized shear}. A generalized shear is a {\em shear} if $ v_1, v_2 \in \mathbb{R},$ $\Lambda$ is represented by a real matrix, and $f$ is slice--regular.
We denote the class of generalized shears as $\mathcal{S}_{\mathbb{H}}$ and the class of shears as $\mathcal{S}_{\mathbb{R}}$.

Analogously, a mapping of the form
$$
  (z,w) \rightarrow (z,w) + (v_1,v_2)(\bar{v}_1 z + \bar{v}_2 w)(e^{ f(\Lambda(z,w))} - 1),
$$   $f \in {\mathcal H},$  is called a {\em generalized overshear}. A generalized overshear is an {\em overshear} if $
v_1, v_2 \in \mathbb{R},$  $\Lambda$ is represented by a real matrix, and $f$ is slice--regular.
We denote the class of generalized overshears as $\mathcal{O}_{\mathbb{H}}$ and the class of overshears as $\mathcal{O}_{\mathbb{R}}$.

\end{definition}
For each fixed $t$ the flows of (generalized) shear or overshear
vector fields are (generalized) shears or overshears.
\begin{lemma}
(Generalized) shears and overshears are time one maps of complete flows and therefore are automorphisms with (generalized) shears and overshears as inverses.
\end{lemma}
\begin{proof}
 The generalized shear $F(z,w) = (z,w) + (v_1,v_2) f(\Lambda(z,w))$ is a flow of the vector field $(v_1,v_2) f(\Lambda(z,w))$ with the flow
$\Phi^X_t(z,w) = (z,w) + (v_1,v_2) t f(\Lambda(z,w)).$
  Similarly, the generalized overshear $G(z,w) =(z,w) + (v_1,v_2)(\bar{v}_1 z + \bar{v}_2 w)(e^{ f(\Lambda(z,w))} -1)$ is a time-one map of the vector field
  $Y(z,w)=(v_1,v_2)(\bar{v}_1 z + \bar{v}_2 w)f(\Lambda(z,w))$ with the flow
  $\Phi^Y_t(z,w) =(z,w) + (v_1,v_2)(\bar{v}_1 z + \bar{v}_2 w)(e^{t f(\Lambda(z,w))} -1).$
\end{proof}


\subsection{Derivatives of shears and overshears.}
Consider a shear $F^{\mu}(z,w) = (z,w) + (\mu,1) f(z - \mu w),$
with $f\in\mathcal{H}^1[u]$. Then, using the notation as in (\ref{diamond}), we have
$$
  DF^{\mu}(z,w)[h_1,h_2]:= \begin{bmatrix}
   h_1+\mu \whp_zf(z-\mu w)[h_1] & \mu \whp_w f(z-\mu w)[h_2]  \\
   \whp_z f(z-\mu w)[h_1] & h_2+\whp_w f(z-\mu w)[h_2]
          \end{bmatrix}.
$$ We would like to calculate the Jacobian, i.e. the Dieudonn\'{e} determinant of
  the above matrix and see if it is - as in the complex or real case -
  proportional to $h_1 h_2$ with  constant factor $1$.
  We may assume that
  $|h_1|=|h_2| = 1$ because of real linearity. Since Gaussian elimination
  of rows by using left multiplication is allowed and $\mu$ is real, we
  have (by a slight abuse of notation we write $\det_D$ also for the
  representative in the quotient)
\begin{eqnarray*}
  \det_D DF^{\mu}(z,w)[h_1,h_2] &=&  \left|
  \begin{matrix}
   h_1 & -\mu h_2  \\
   \whp_z f(z-\mu w)[h_1] & h_2-\mu \whp_z f(z-\mu w)[h_2]
   \end{matrix} \right|\\
   & =& h_1 h_2 - \mu h_1 \whp_z f(z-\mu w)[h_2] + \mu h_1 \whp_z f(z-\mu w)[h_1] (h_1)^{-1}h_2.
\end{eqnarray*}
The last two terms do not cancel out in general, but they do if $h_1 =
h_2.$ Therefore we could say that for $|h| = 1$ the determinant $\det_D
DF^j(z,w)[h,h] = 1,$ which means, that shears could be considered in a
way as volume preserving maps.
However, this property is no longer
preserved if we compose two shears or if $\mu$ is not real.

For instance, let $f(u) = u^2$ and consider $F^{\mu}$ as above. Recall that $\det_D A = 1$ precisely when
its representative has modulus $1.$
Even if we simplify the calculation by inserting $h = 1,$ we get
\begin{eqnarray*}
  \det_D DF^{\mu}(z,w)[1,1]&=& \left| \begin{matrix}
   1  & -\mu   \\
   2(z - \mu w)    & 1- (\mu(z - \mu w) + (z - \mu w)\mu)
          \end{matrix}\right| =
           \\
   &=& 1- (\mu(z-\mu w) - (z - \mu w)\mu).
\end{eqnarray*}
\noindent The number in the bracket is purely imaginary and so the
only possibility for such a number to have modulus $1,$ is, that the
term in the bracket vanishes for all $z$ and $w.$ This is iff $\mu
\in \mathbb{R}.$

Assume $\mu $ is real; in order to calculate the  derivative of the overshear flow
$$\Phi^Y(z,w) = (z,w) + (\mu,1)(\mu^2 + 1)^{-1}(\mu z+w) (e^{tf(z - \mu w) }-1)$$
of the vector field
$$Y(z,w) = (\mu,1)(z\mu + w)f(z - \mu w)(\mu^2 + 1)^{-1}$$ we notice first  that
$$\whp_we^{f(z - \mu w)}[h] = -\mu  \whp_ze^{tf(z - \mu w)}[h]$$
and then put
$$ (-\mu) A:=\whp_we^{f(z - \mu w)}[h] = -\mu  \whp_ze^{tf(z - \mu w)}[h] \quad B:= e^{tf(z - \mu w) }-1.$$ Then,
$$
  D\Phi^Y(z,w)[h,h]:= \begin{bmatrix}
   h+\frac{\mu}{\mu^2 + 1}(\mu h B + (z\mu+w)A)  &  \frac{\mu}{\mu^2 + 1}( h B - \mu (z\mu+w)A) \\
   \frac{1}{\mu^2 + 1}(\mu h B + (z\mu+w)A) & h+ \frac{1}{\mu^2 + 1}( h B - \mu (z\mu+w)A)
          \end{bmatrix}.
$$
After applying Gaussian elimination on rows, we see that
$$
  \det_D \Phi^Y(z,w)[h,h]= h  \left| \begin{matrix}
   1  & -\mu   \\
    \dfrac{\mu h B + (z\mu +w)A}{\mu^2 + 1}   & h+\dfrac{\mu h B - \mu (z\mu +w)A}{\mu^2 + 1}
          \end{matrix}\right| = h^2 e^{tf(z - \mu w)},
$$
  so we can say that the Dieudonn\'e
  determinant of  $\Phi^Y(z,w)$ is represented by
  the function  $V(z,w;t)=e^{tf(z - \mu w)}$ and in this case the function $V(z,w;t)$
  also solves the differential equation
$$
  \frac{d}{dt} V(z,w,t) = f(z - \mu w)V(z,w,t) ,  \; V(z,w,0) = 1,
$$
where $\div Y (z,w) = f(z - \mu w).$ Therefore we can say that overshears form a class of automorphisms which resemble the property of having volume and the quantity
$V$ resembles the volume at $\Phi^Y(z,w,t).$

\section{Andersen--Lempert theorem for automorphisms with volume}

As shown in the previous section any notion of volume and of volume-preserving maps are not well--defined in general
if one uses a definition which involves the notion of the determinat.
Therefore we prefer to use another approach and, as for
 the case automorphisms of $\mathbb{C}^n$, we consider
the volume-preserving automorphisms  to
be  those which are perturbations of the identity by vector fields with divergence.
\begin{definition}

  The space of {\em automorphisms with volume} is defined as
$$
  \Aut_V (\mathbb{H}^2) = \lbrace \Phi^X(z,w,1), \Div X(z,w)[h] = h \div X(z,w) \rbrace
$$
where $X$ is  a vector field with corresponding flow $\Phi^X$.
The space of {\em automorphisms with volume $1$} is defined as
$$
  \Aut_1 (\mathbb{H}^2) = \lbrace \Phi^X(z,w,1), \Div X(z,w)[h] = 0 \rbrace.
$$
\end{definition}
\noindent Examples in the previous sections show the remarkable fact  that

\[ \mathcal{S}_{{\mathbb{R}}} \subset  \Aut_1 (\mathbb{H}^2)\]

  but
\[ \mathcal{S}_{\mathbb{H}} \not\subset  \Aut_1 (\mathbb{H}^2).\]
  Similar conclusions hold for overshears and  generalized overshears.

\begin{example} { In the complex case for every automorphism $F(z,w) = (z,w) + h.o.t.,$ there is vector field
$X$ defined by the flow $\Phi(z,w,t) = F(tz,tw)/t.$ If $F$ is volume preserving,
 then $\div X = 0.$ The same holds for a composition of two automorphisms $F$ and $G$ and a corresponding associated flow.
This no longer holds true in the quaternionic case. After composing the
shears $F(z,w) = (z, w + z^2) $ and $G(z,w) = (z + w^2,w)$, one can
define as corresponding flow the mapping
$$
  \Phi(z,w,t) = F\circ G(tz,tw)/t = (z,w) + t(w^2,z^2)  + t^2(0, zw^2 + w^2 z) + t^3(0,w^4).
$$
The equation $d/dt (\Phi(z,e,t)) = X(\Phi(z,w,t),t)$ defines the time-dependent vector field $X(z,w,t) = \sum_0^{\infty} X_n(z,w) t^n.$
If the vector field $X$ is supposed to have  divergence, then all the vector fields $X_n$ should have divergence $0$, in particular, they should be bidegree full.
The defining equation in our case is then
$$
  (w^2,z^2)  + 2 t (0, zw^2 + w^2 z) + 3t^2(0,w^4) = \sum_0^{\infty} X_n(z + tw^2,w + tz^2 + t^2 (zw^2 + w^2 z) + t^3 w^4 ) t^n
$$ which by identity principle on $t$ implies $X_0(z,w) = (w^2,z^2),$
$X_1(z,w) + (w z^2 + z^2 w, z w^2 + w^2 z) = 2 (0, zw^2 + w^2 z).$ The
vector field $X_1(z,w) = (w z^2 + z^2 w,-z w^2 - w^2 z)$ is not
BF. Notice, that we do not claim that there does not exist another
divergence zero vector field $Y$ with the flow $\Phi^Y_t$ such that $F
\circ G = \Phi^Y_1.$ Therefore, we remark that in general a  finite composition of shears is an automorphism but
not necessarily a map with volume 1. In other words, it is possible
that a sufficiently small neighborhood of a finite composition of shears does not contain any other
  composition of shears.}

\end{example}
Having said that,  the following theorem is a direct application of the classical Ander-\-sen--Lempert theory as developed in \cite{AL}.
\begin{theorem}
  Every automorphism in   $\Aut_1 (\mathbb{H}^2)$
  can be approximated uniformly on
  compacts by finite composition of shears and overshears and every
  automorphism with volume $1$ can be approximated uniformly on
  compacts by a finite composition of shears.
\end{theorem}

\example In this example we show that the map $F(z,w) = (z e^{zw},
e^{-zw}w)$ from  example \ref{TheExample} is not approximable by
finite compositions of shears. It is, though,  a time one map of a
complete vector field, but this vector field does not have 
divergence.

The Taylor expansion of the mapping $F$ is of the form
$$
  F(z,w) = (z + z^2 w + \ldots, w - zw^2 + \ldots),
$$
where the dots indicate higher order terms.
Consider a generic composition of shears
$S = S^d \circ \ldots \circ S^1$
with
$$S^m(z,w) = (z,w) + (\mu_m,1)((z - \mu_m w)^2 a_{m,2} + (z - \mu_m w)^3 a_{m,3} + \ldots)$$
and let $S_n^m$ denote the term of order $n$ in its expansion.
Then the composition of shears $S$ up to the third order is of the form
$$id + \sum_{m=1}^k S^m_2 + \sum_{m=1}^k S_3^m + \tilde{S}_3,$$
where $\tilde{S}_3$ are  the rest of the terms of order $3$.
If $S$ is supposed to be approximating $F$, the terms of order $3$ of $S$ should approximate the term of order $3$ in the expansion of $F$ - the term
$(z^2w, - zw^2).$ Since the terms $S^m_n$ are all BF and the latter is not, the only possibility for approximating $F$  is that the missing terms come out from $\tilde{S}_3.$ However, terms of order $3$ arise  iff we compose some $S^m_2$ with a term of the form
$\id + T_2$ where $T_2$ are terms of order $2,$ which are all BF. So we have
\begin{eqnarray*}
   && (z - \mu_m w)^2 \circ ((z,w) + \sum_{n}(z - \mu_n w)^2(\mu_n,1) a_{n}) = \\
   &&\quad= ((z - \sum_n(z - \mu_n w)^2 \mu_n a_{n}) - \mu_m (w + \sum_n(z - \mu_n w)^2 a_{n}))\\
  &&\quad= ((z - \mu_m w) + \sum_n (z - \mu_n w)^2(\mu_n-\mu_m) a_{n})^2\\
  &&\quad= (z - \mu_m w)^2 + (z - \mu_m w)\sum_n(z - \mu_n w)^2(\mu_n-\mu_m) a_{n} +\\
  &&\quad  \quad+ \sum_n(z - \mu_n w)^2(\mu_n-\mu_m) a_{n} (z - \mu_m w) + \ldots \\
  &&\quad= (z - \mu_m w)^2 + \sum_n (\mu_n - \mu_m)[(z - \mu_m w) (z - \mu_n w)^2 a_n + (z - \mu_n w)^2a_n(z - \mu_m w)].
\end{eqnarray*}
We are interested in the terms in the square brackets with bidegree $(2,1).$ Those are
\begin{eqnarray*}
  &&(- \mu_m w z^2 -\mu_n z (z w + w z) )a_n + (-\mu_m z^2 a_n w - \mu_n(zw + wz) a_n z)\\
  &&\quad= -(w z^2 \mu_m a_n + z^2  w \mu_n a_n + zwz\mu_n a_n)  - (z^2 \mu_m a_n w + z w \mu_n a_n z + wz\mu_n a_n z).
\end{eqnarray*}
After summing up all possible choices we get 
\begin{eqnarray*}
   -w z^2 (\sum_n \mu_m a_n) - (z^2 w + zwz) ( \sum_n \mu_n a_n) +\\- z^2 (\sum_n \mu_m a_n)w - zw ( \sum_n \mu_n a_n)z - wz ( \sum_n \mu_n a_n)z.
\end{eqnarray*}
The bidegree full part can cancel out only terms with coefficients on the right. So if the above--given sums are not real, we can not get rid of the terms
$zw( \sum_n \mu_n a_n)z$ and $z^2 (\sum_n \mu_m a_n)w $. On the other hand, if the sums are real, we can rewrite the above expression as
$$
  ((\sum_n \mu_m a_n)-(\sum_n \mu_n a_n))(wz^2 + z^2w) - 2(\sum_n \mu_m a_n) zwz.
$$
We observe that  bidegree polynomials with degree $d=3$
can not cancel out the term $w z^2$ in the first component of the mapping without cancelling also the term $z^2w.$
So, the conclusion is, that $F$ cannot be approximated by a composition of shears.
Finally, we remark that in the above considerations,
the monomials $wz^2, zwz$ and $ z^2w$ are not just formally linearly independent, but also linearly independent as functions.

\end{document}